\documentclass[11pt]{amsart}

\usepackage[T1]{fontenc}
\usepackage{lmodern}
\usepackage{microtype}
\usepackage{comment}
\usepackage{amsmath,amssymb,mathtools}
\usepackage{enumitem}
\usepackage{xcolor}
\usepackage[colorlinks=true,linkcolor=blue!55!black,citecolor=blue!55!black,
            urlcolor=blue!55!black]{hyperref}
\usepackage[nameinlink,noabbrev]{cleveref}

\numberwithin{equation}{section}

\newtheorem{theorem}{Theorem}[section]
\newtheorem{proposition}[theorem]{Proposition}
\newtheorem{lemma}[theorem]{Lemma}
\newtheorem{corollary}[theorem]{Corollary}
\newtheorem{remark}[theorem]{Remark}
\newtheorem{example}[theorem]{Example}
\theoremstyle{definition}

\DeclareMathOperator{\AC}{AC}
\DeclareMathOperator{\Dom}{Dom}
\newcommand{\R}{\mathbb R}
\newcommand{\E}{\mathbb E}
\newcommand{\N}{\mathbb N}

\newcommand{\cE}{\mathcal E}

\newcommand{\cV}{\mathcal V}
\newcommand{\one}{\mathbf 1}
\newcommand{\loc}{\mathrm{loc}}

\newcommand{\dd}{\mathop{}\!\mathrm{d}}

\title[One-dimensional Feynman--Kac regularity]
{One-dimensional Feynman--Kac regularity under the
Engelbert--Schmidt conditions}

\author{Johannes Ruf}
\address{Department of Mathematics, London School of Economics and Political Science, London, United Kingdom}
\email{j.ruf@lse.ac.uk}
\date{}


\begin{document}

\begin{abstract}
Let $X$ be a time-homogeneous one-dimensional diffusion whose
coefficients satisfy the Engelbert--Schmidt conditions, and let $q$ be
a bounded potential.  We show that the associated Feynman--Kac
semigroup is smooth in time and $C^1$ in space, with a locally absolutely
continuous first spatial derivative; the Kolmogorov equation holds
almost everywhere.  If the drift, second-order coefficient, and potential
are continuous, the Feynman--Kac value is $C^{1,2}$ in the interior.
Thus, in the time-homogeneous one-dimensional setting, continuity
suffices for interior $C^{1,2}$ regularity.  We identify the
corresponding semigroups with the killed and reflected Feynman--Kac
functionals, treat nonzero, time-independent  Dirichlet data, and give examples
showing the sharpness of the continuity assumptions and the failure of
the corresponding statement in two dimensions.
\end{abstract}

\maketitle

\section{Introduction}

Consider, on an open interval $I=(\ell,r)$, the time-homogeneous stochastic differential equation \begin{equation}\label{eq:sde-intro} \dd X_t=b(X_t)\,\dd t+\sigma(X_t)\,\dd W_t, \qquad X_0=x\in I, \end{equation} up to its exit time $\zeta$ from $I$. We study the spatial regularity of the associated Feynman--Kac value under the low regularity allowed by the Engelbert--Schmidt conditions.

Classical Feynman--Kac and parabolic regularity results are often
formulated under H\"older or stronger spatial regularity assumptions;
see, for example, the classical treatment in \cite[Chapter~6]{Friedman1975}.  Similar assumptions
appear in the mathematical-finance literature.  Janson and Tysk
\cite{JansonTysk2006} assume local Lipschitz continuity of the diffusion
coefficients in their multidimensional setting and local
$1/2$-H\"older continuity in one dimension.  Ekstr\"om and Tysk
\cite[Hypothesis~2.1 and Theorem~3.2]{EkstromTysk2009} and Bayraktar
and Xing \cite[Theorem~2]{BayraktarXing2010} likewise work under local
$1/2$-H\"older continuity in space.  Karatzas and Ruf
\cite[Propositions~5.2 and 5.4]{KaratzasRuf2016} characterize the tail
distribution of the explosion time of a one-dimensional diffusion as
the minimal nonnegative classical solution of the associated Cauchy
problem under local H\"older assumptions.
For the natural-scale half-line problem, Chen, Huang, Song, and Zhu
\cite[Theorems~4.4 and~5.2 and  Remark~5.4]{ChenHuangSongZhu2017}
obtain classicality under local $\delta$-H\"older continuity of $\sigma$
for any $\delta>0$ and, under mere continuity, represent the
stochastic solution as a limit of classical approximations.

Ekstr\"om, Janson, and Tysk
\cite[Theorem~8.1 and Corollary~8.3]{EkstromJansonTysk2015}
consider natural-scale generalized diffusions on $\R$.  For continuous
initial data satisfying a global growth condition, they prove
smoothness in time and show that absolute continuity of the speed
measure yields $C^1$ spatial regularity, while continuity of its density
yields $C^2$ spatial regularity.  Consequently, for a natural-scale diffusion with zero potential, their results already give
positive-time $C^{1,2}$ regularity when $\sigma$ is
continuous and nonzero.  Compared with this natural-scale, zero-potential result, we allow
general Engelbert--Schmidt coefficients before passage to scale,
bounded Borel potentials, arbitrary intervals, and the canonical
killed and reflected realizations.  Within the $L^2$ framework, the
initial datum is arbitrary and need not be continuous or satisfy a
pointwise growth condition.

We work under the Engelbert--Schmidt conditions
\[
  \sigma\ne0,
  \qquad
  \frac1{\sigma^2},\ \frac b{\sigma^2}\in L^1_{\loc}(I),
\]
which ensure weak existence and uniqueness in law up to the exit time;
see \cite[Theorem~5.5.15]{KaratzasShreve1991} and
\cite{EngelbertSchmidt1985}.  Our argument passes to natural scale and uses the associated symmetric
$L^2$ operator.  For positive times, the semigroup maps $L^2$ data into
functions that are $C^1$ in space with locally absolutely continuous
first derivative, and the Kolmogorov equation holds almost everywhere.
If the drift, second-order coefficient, and bounded potential are
continuous, this yields interior $C^{1,2}$ regularity.

We also consider boundary behavior.  The zero-boundary form domain corresponds to the killed diffusion, while the maximal form
domain gives the symmetric reflected extension; see
\cite{FukushimaOshimaTakeda2010,Fukushima2010,Fukushima2014}.
We identify the corresponding semigroups with their Feynman--Kac
representations and also treat nonzero, time-independent  Dirichlet data by an
affine function of the scale variable.

\Cref{ex:jump-coefficients}  shows that none of the separate continuity
hypotheses on $a=\sigma^2/2$, $b$, and the potential $q$ can be dropped without replacement:
a discontinuity in any one of these coefficients can destroy twice
differentiability of the Feynman--Kac value.  In
addition, using the construction in
\cite[proof of Theorem~4]{EscauriazaMontaner2017}, we give a
time-homogeneous two-dimensional example with $b=q=0$ and a continuous
uniformly elliptic coefficient matrix for which the stationary
Feynman--Kac value is not $C^2$.

Regularity is separate from uniqueness.  On unbounded intervals,
additional growth or integrability conditions may be needed for
uniqueness among classical solutions.  Strict local martingales provide
standard examples in which several classical solutions have the same
terminal data; see
\cite{CoxHobson2005,EkstromTysk2009,BayraktarXing2010,Ruf2013}.

The paper is organized as follows.  \Cref{sec:ES} introduces the
scale--speed and $L^2$ framework and proves the two
regularity lemmas used below.  \Cref{sec:main} constructs the semigroup
with bounded Borel potential, proves its interior regularity, records
the Dirichlet and Neumann-type boundary behavior, and treats nonzero, 
time-independent Dirichlet data.  \Cref{sec:FK} gives the corresponding
Feynman--Kac representations for the killed and reflected diffusions
and records the classical verification argument.  Finally,
\Cref{sec:sharpness} gives an example showing the
sharpness of the continuity assumptions and a two-dimensional example
showing that, already in dimension two, continuity and uniform
ellipticity of the diffusion matrix do not imply interior $C^{1,2}$
regularity.

\section{Engelbert--Schmidt diffusions and the symmetric framework}
\label{sec:ES}

We first recall the Engelbert--Schmidt conditions, the transformation to
natural scale, and the $L^2$
framework.  We then introduce the form domains and prove the two regularity lemmas used throughout the paper.

Throughout, $I=(\ell,r)$ is a nonempty open interval, possibly with infinite
endpoints.  We impose the Engelbert--Schmidt conditions
\begin{equation}\label{eq:ES}
  \sigma(x)\ne0\quad(x\in I),
  \qquad
  \frac{1}{\sigma^2},\ \frac{b}{\sigma^2}\in L^1_{\loc}(I),
\end{equation}
where $b$ and $\sigma$ are Borel functions.  Equivalently,
\begin{equation*}
  a>0,\qquad \frac1a,\ \frac ba\in L^1_{\loc}(I),
  \qquad a=\frac{\sigma^2}{2}.
\end{equation*}
Under \eqref{eq:ES}, for every $x\in I$, equation
\eqref{eq:sde-intro} admits a weak solution that is unique in law up to its
possibly finite exit time $\zeta$ from $I$; see
\cite[Theorem~5.5.15]{KaratzasShreve1991}.

Fix $x_0\in I$.  Define the scale density and scale function by
\begin{equation*}
  \rho(x)=
  \exp\!\left(-\int_{x_0}^x\frac{b(z)}{a(z)}\,\dd z\right),
  \qquad
  s(x)=\int_{x_0}^x\rho(z)\,\dd z.
\end{equation*}
Then $\rho$ is strictly positive and belongs to $\AC_{\loc}(I)$,
while $s\in C^1(I)$ is strictly increasing with $s'=\rho$, and
$\rho'=-(b/a)\rho$ almost everywhere.
Set now $J=s(I)=(\alpha,\beta)$, where either endpoint may be infinite.
By the standard scale transformation for one-dimensional diffusions
\cite[Section~5.5.B]{KaratzasShreve1991}, the process $Y_t=s(X_t)$, $0\le t<\zeta$, satisfies
\begin{equation*}
  \dd Y_t=\widetilde\sigma(Y_t) \dd W_t,
  \qquad
  \widetilde\sigma(s(x))=\rho(x)\sigma(x),
\end{equation*}
and is therefore in natural scale, with the same lifetime $\zeta$.  Its  coefficient $\widetilde\sigma$ again satisfies
the corresponding Engelbert--Schmidt condition.

 We use the
symmetrizing measure
\begin{equation*}
  \mu(\dd x)=\frac{\dd x}{a(x)\rho(x)}
          =\frac{2\,\dd x}{\sigma^2(x)\rho(x)}.
\end{equation*}
Since $\rho$ is strictly positive and continuous and
$1/a\in L^1_{\loc}(I)$, the measure $\mu$ is locally finite and has
full support on $I$.
Its push-forward by $s$ is $w(y)\,\dd y$, where
\begin{equation}\label{eq:w}
  w(s(x))=\frac{1}{a(x)\rho^2(x)}
         =\frac{2}{\sigma^2(x)\rho^2(x)}.
\end{equation}
Hence $w$ is strictly positive and belongs to $L^1_{\loc}(J)$.

For the symmetric $L^2$ formulation, let
$
  H=L^2(J,w(y)\,\dd y)
$
and define
\begin{align} \label{eq:Vmax}
  \cV_{\max}
  =
  \left\{
    v\in H:
    v\text{ has an }\AC_{\loc}(J)\text{ representative with }
    v'\in L^2(J)
  \right\}.
\end{align}
Since $w>0$, this representative is unique, and we identify $v$ with it.
We call a linear subspace $\cV\subset\cV_{\max}$ admissible if
$
  C_c^\infty(J)\subset\cV
$
and $\cV$ is complete under the norm
\[
  \|v\|_{\cV}^2
  =
  \|v\|_H^2+\|v'\|_{L^2(J)}^2.
\]
In particular, $\cV_{\max}$ itself is admissible. 
For any admissible $\cV$, define
\begin{equation*}
  \cE(v,z)=\int_J v'(y)z'(y)\,\dd y,
  \qquad v,z\in\cV.
\end{equation*}
Then $(\cE,\cV)$ is a densely defined, closed, nonnegative symmetric
form on $H$.  By the representation theorem for closed symmetric forms
\cite[Theorem~1.3.1 and Corollary~1.3.1]{FukushimaOshimaTakeda2010},
there is a unique nonpositive self-adjoint operator $A$ on $H$
characterized by
\begin{equation}\label{eq:generator-form}
  \cE(v,z)=-\langle Av,z\rangle_H,
  \qquad v\in\Dom(A),\ z\in\cV.
\end{equation}

For certain admissible choices of $\cV$, $(\cE,\cV)$ is a regular
Dirichlet form and hence defines a symmetric Markov process.  These
processes are symmetric extensions of the minimal diffusion and agree
with it up to the first boundary hit; see
\cite{Fukushima2010,Fukushima2014}.  The regularity argument below does
not use this Markovian structure.
For the probabilistic Feynman--Kac identifications in \Cref{sec:FK},
we shall use the two canonical form domains $\cV_0$ and
$\cV_{\max}$. The domain $\cV_{\max}$ was defined in \eqref{eq:Vmax}, while
\begin{equation*}
  \cV_0
  =
  \left\{
    v\in\cV_{\max}:
    v(\alpha)=0\ \text{if }\alpha>-\infty,\
    v(\beta)=0\ \text{if }\beta<\infty
  \right\}.
\end{equation*}
Here the boundary values at finite endpoints of $J$ are understood as
limits from within $J$, which exist because $v'\in L^2(J)$. The space
$\cV_0$ is also admissible. In \Cref{sec:FK}, the semigroups associated with $\cV_0$ and
$\cV_{\max}$ are represented probabilistically by the killed and
reflected diffusions, respectively.

We use two elementary lemmas.  The first gives local spatial regularity
for functions in the operator domain; the second shows that the
semigroup maps $H$-data into the operator domain at positive times.

\begin{lemma}[Operator-domain regularity]\label{lem:domain}
Let $J$ be an open interval and let $w:J\to(0,\infty)$ be Borel with
$w\in L^1_{\loc}(J)$.  Let $\cV$ be admissible in
$H=L^2(J,w(y)\,\dd y)$, and let $\cE$ and $A$ be as above.
Then every  $v\in\Dom(A)$ belongs to $C^1(J)$, with
$v'\in\AC_{\loc}(J)$, and
\begin{align} \label{eq:260820}
  v''=wAv,
  \quad\text{almost everywhere}.
\end{align}
Moreover, for every compact interval $K\subset J$, the restriction map
$\Dom(A)\to C^1(K)$, where $\Dom(A)$ is equipped with the graph norm,
is continuous.
\end{lemma}
\begin{proof}
Let $v\in\Dom(A)$.  Since
$
  \Dom(A)\subset\cV\subset\cV_{\max}
$, 
the function $v$ is locally absolutely continuous and
$v'\in L^2(J)$.  Moreover, since
$C_c^\infty(J)\subset\cV$, \eqref{eq:generator-form} gives, for every
$\varphi\in C_c^\infty(J)$,
\[
  \int_J v'(y)\varphi'(y)\,\dd y
  =
  -\int_J Av(y)\varphi(y)w(y)\,\dd y.
\]
The right-hand side is locally integrable, since for every compact interval
$K_0\subset J$,
\[
  \int_{K_0}|Av(y)|w(y)\,\dd y
  \le
  \left(\int_{K_0}|Av(y)|^2w(y)\,\dd y\right)^{1/2}
  \left(\int_{K_0}w(y)\,\dd y\right)^{1/2}.
\]
Thus the distributional derivative of $v'$ is $wAv\in L^1_{\loc}(J)$.
Consequently, $v'\in\AC_{\loc}(J)$, hence $v\in C^1(J)$, and
\eqref{eq:260820} holds.

  Choose now a nondegenerate
compact interval $K_0\subset J$ containing $K$ and set
$
  M_0=\int_{K_0}w(y)\,\dd y>0$.
For $x,y\in K_0$, Cauchy--Schwarz gives
\[
  |v(x)|
  \le |v(y)|+|K_0|^{1/2}\|v'\|_{L^2(K_0)}.
\]
Integrating in $y$ against $w(y)\,\dd y/M_0$ and applying
Cauchy--Schwarz once more yields
\begin{align} \label{eq:260823.1}
  \|v\|_{L^\infty(K)} \leq \|v\|_{L^\infty(K_0)}
  \le
  M_0^{-1/2}\|v\|_H
  +|K_0|^{1/2}\|v'\|_{L^2(J)}.
\end{align}
On the other hand, taking $z=v$ in \eqref{eq:generator-form} gives
\[
  \|v'\|_{L^2(J)}^2
  =\cE(v,v)
  =-\langle Av,v\rangle_H
  \le \|Av\|_H\|v\|_H,
\]
and hence
\begin{align} \label{eq:260823.2}
  \|v'\|_{L^2(J)}
  \le \frac12\bigl(\|Av\|_H+\|v\|_H\bigr).
\end{align}
Finally, since $v'\in\AC(K_0)$, we may choose
$y_0\in K_0$ such that
$
  |v'(y_0)|
  \le|K_0|^{-1/2}\|v'\|_{L^2(K_0)}$.
Hence, for $x\in K_0$,
\begin{align*}
  |v'(x)|
  &\le
  |K_0|^{-1/2}\|v'\|_{L^2(J)}
  +\int_{K_0}|v''(y)|\,\dd y,
\end{align*}
yielding
\begin{align} \label{eq:260823.3}
    \|v'\|_{L^\infty(K)} \leq \|v'\|_{L^\infty(K_0)}
   &\le
  |K_0|^{-1/2}\|v'\|_{L^2(J)}
  +M_0^{1/2}\|Av\|_H.
\end{align}
Combining \eqref{eq:260823.1} and \eqref{eq:260823.3} with \eqref{eq:260823.2} proves the existence of $C_K>0$
such that
\begin{equation*}
  \|v\|_{C^1(K)}
  \le C_K\bigl(\|v\|_H+\|Av\|_H\bigr),
  \qquad v\in\Dom(A),
\end{equation*}
hence the asserted continuity follows.
\end{proof}

\begin{lemma}[Semigroup regularization into operator domains]
\label{lem:semigroup-domain}
Let $H$ be a Hilbert space and let $L$ be a self-adjoint operator on
$H$ that is bounded above.  For $g,h\in H$, define
\[
  z(t)=e^{tL}g-\int_0^t e^{rL}h\,\dd r.
\]
Then
$
  z\in C^\infty((0,\infty);\Dom(L))
$,
where $\Dom(L)$ is equipped with its graph norm, and
\begin{align}\label{eq:260822.2}
  z'(t)=Lz(t)-h,
  \qquad t>0.
\end{align}
\end{lemma}

\begin{proof}
Since $L$ is bounded above, for every $f\in H$, $t>0$, and $k\in\N$,
$
  \sup_{\lambda\in\sigma(L)}
  |\lambda|^k e^{t\lambda}<\infty$.
Hence the spectral theorem gives
$
  e^{tL}f\in\Dom(L^k)
$.
Fix now $f\in H$, $k\in\N_0$, and $t_0>0$. 
For $t>t_0/2$, we have
\[
  L^k e^{tL}f
  =
  e^{(t-t_0/2)L}L^k e^{(t_0/2)L}f.
\]
The strong continuity of $(e^{tL})_{t\ge0}$ therefore shows that
$t\mapsto L^k e^{tL}f$ is continuous in $H$ on $(0,\infty)$.

Next, for $\varepsilon>0$, the semigroup property gives
\[
  \frac{e^{\varepsilon L}-\operatorname{Id}_H}{\varepsilon}\int_0^t e^{rL}h\,\dd r
  =
  \frac1\varepsilon\left(
    \int_t^{t+\varepsilon}e^{rL}h\,\dd r
    -
    \int_0^\varepsilon e^{rL}h\,\dd r
  \right).
\]
By strong continuity, the right-hand side converges in $H$ to
$e^{tL}h-h$ as $\varepsilon\downarrow0$.  Since $L$ is the generator
of $(e^{tL})_{t\ge0}$, it follows that
\[
  \int_0^t e^{rL}h\,\dd r\in\Dom(L),
  \qquad
  L\int_0^t e^{rL}h\,\dd r=e^{tL}h-h;
\]
in particular, $z(t)\in\Dom(L)$.

As an $H$-valued function, $z$ is differentiable for $t>0$, with
\[
  z'(t)=Le^{tL}g-e^{tL}h=Lz(t)-h.
\]
This proves \eqref{eq:260822.2}.
Iterating gives, for every $m\in\N$,
\begin{align} \label{eq:260825.1}
  z^{(m)}(t)
  =
  L^m e^{tL}g-L^{m-1}e^{tL}h.
\end{align}
By the spectral estimate and the continuity established above, \eqref{eq:260825.1} shows that, for every $m\in\N$, $z^{(m)}(t)\in\Dom(L)$ and
both $z^{(m)}$ and $Lz^{(m)}$ are continuous as $H$-valued
functions.  For $m=0$, \eqref{eq:260822.2} gives $Lz=z'+h$, which
is continuous in $H$.  Thus $z^{(m)}$ is graph-norm continuous for
every $m\in\N_0$. Finally, \eqref{eq:260825.1}  and \eqref{eq:260822.2} show that $Lz^{(m)}$ is
differentiable in $H$ with derivative $Lz^{(m+1)}$ for every
$m\in\N_0$.  Since $z^{(m)}$ has $H$-valued derivative
$z^{(m+1)}$, it has graph-norm derivative $z^{(m+1)}$.  Hence
$z\in C^\infty((0,\infty);\Dom(L))$.
\end{proof}

\section{Regularity and boundary conditions with bounded potentials}
\label{sec:main}

In this section, we construct the $L^2$ semigroup with bounded
potential and establish its interior regularity.   We then identify the boundary behavior
for the canonical Dirichlet and Neumann-type conditions and finally
treat nonzero, time-independent Dirichlet data.

Fix an admissible form domain $\cV$ as in \Cref{sec:ES}, and let $A$
be its associated operator on
$
  H=L^2(J,w(y)\,\dd y)
$.
Since the push-forward of $\mu$ under the scale function $s$ is
$w(y)\,\dd y$, the map
\[
  U:L^2(I,\mu)\longrightarrow H,
  \qquad
  Uf=f\circ s^{-1},
\]
is unitary.
Let $q:I\to\mathbb R$ be bounded and Borel and let $Q$ denote the bounded self-adjoint
multiplication operator on $H$ by $q\circ s^{-1}$.

The semigroup is obtained by the following bounded self-adjoint
perturbation.

\begin{proposition}[$L^2$ semigroup with bounded potential]
\label{prop:L2-semigroup}
Assume \eqref{eq:ES}, let $\cV$ be admissible with associated operator
$A$, and let $q:I\to\mathbb R$ be bounded and Borel.  Let $Q$ be
multiplication on $H$ by $q\circ s^{-1}$.  Then
\[
  B=A-Q,\qquad \Dom(B)=\Dom(A),
\]
is self-adjoint and bounded above by $\|q^-\|_\infty$.  Moreover,
\begin{equation}\label{eq:L2-conjugation}
  T_t^q=U^{-1}e^{tB}U,\qquad t\ge0,
\end{equation}
defines a strongly continuous semigroup on $L^2(I,\mu)$ satisfying
\begin{equation}\label{eq:FK-L2-bound}
  \|T_t^qf\|_{L^2(I,\mu)}
  \le e^{t\|q^-\|_\infty}\|f\|_{L^2(I,\mu)},
  \qquad f\in L^2(I,\mu).
\end{equation}
\end{proposition}

\begin{proof}
Since $Q$ is bounded and self-adjoint and $A$ is self-adjoint,
$B=A-Q$ with $\Dom(B)=\Dom(A)$ is self-adjoint.  Moreover, since
$\langle Av,v\rangle_H=-\cE(v,v)\le0$ and
$-\langle Qv,v\rangle_H\le\|q^-\|_\infty\|v\|_H^2$,
$B$ is bounded above by $\|q^-\|_\infty$.

Hence the spectral theorem shows that $(e^{tB})_{t\ge0}$ is a
strongly continuous semigroup on $H$ and
\begin{align} \label{eq:260823.4}
  \|e^{tB}h\|_H
  \le e^{t\|q^-\|_\infty}\|h\|_H,
  \qquad h\in H.
\end{align}
Since $U$ is unitary, \eqref{eq:L2-conjugation} defines a strongly
continuous semigroup on $L^2(I,\mu)$, and \eqref{eq:260823.4} yields
\eqref{eq:FK-L2-bound}.
\end{proof}

We next prove the interior regularity of the semigroup in
\Cref{prop:L2-semigroup}.

\begin{theorem}[Interior regularity under Engelbert--Schmidt conditions]
\label{thm:main}
Assume \eqref{eq:ES}, let $\cV$ be an admissible form domain with
associated operator $A$, let $q:I\to\mathbb R$ be bounded and Borel,
and recall the strongly continuous semigroup $(T_t^q)_{t\ge0}$ from
\Cref{prop:L2-semigroup}. Let $f\in L^2(I,\mu)$.  For every $t>0$, $T_t^qf$ admits a unique
continuous representative, which we identify with $T_t^qf$.
Furthermore, the following statements hold with $u(t,x)=T_t^qf(x)
$.
\begin{enumerate}[label=\textup{(\roman*)},ref=\textup{(\roman*)}]
\item\label{item:main-a}
For every $k\in\N_0$, $\partial_t^ku$ and
$\partial_x\partial_t^ku$ exist on $(0,\infty)\times I$ and are jointly
continuous there. 

\item\label{item:main-b}
For every $t>0$, $\partial_xu(t,\cdot)\in\AC_{\loc}(I)$, and for almost every $x\in I$,
\begin{equation}\label{eq:a.e.-PDE}
  \partial_tu(t,x)
  =
  a(x)\partial_{xx}u(t,x)
  +b(x)\partial_xu(t,x)-q(x)u(t,x).
\end{equation}
For every $k\in\N_0$, the same assertions hold with $u$ replaced by
$\partial_t^ku$.

\item\label{item:main-c}
If $a$, $b$, and $q$ are continuous on $I$, then
\[
  u\in C^{1,2}((0,\infty)\times I)
\]
and \eqref{eq:a.e.-PDE} holds pointwise on $(0,\infty)\times I$.  Moreover,
$\partial_t^ku\in C^{1,2}((0,\infty)\times I)$ for every $k\in\N_0$.
\end{enumerate}
In addition,
$
  u(t,\cdot)\to f
$
in $L^2(I,\mu)$ as $t\downarrow0$.
\end{theorem}

\begin{proof}
Fix $f\in L^2(I,\mu)$ and, for $t>0$, set
$
  v(t)=e^{tB}Uf
$.
Since $B$ is self-adjoint and bounded above, \Cref{lem:semigroup-domain},
applied with $L=B$, $g=Uf$, and $h=0$, gives
$
  v\in C^\infty((0,\infty);\Dom(B))
$
and
$
  v'=Bv
$.
Hence, inductively,
$
  v^{(k+1)}=Bv^{(k)}
$
for every $k\in\N_0$.

Since $B=A-Q$ with $Q$ bounded, the graph norms
of $A$ and $B$ are equivalent.  Therefore, \Cref{lem:domain} implies
that, for every compact interval $K\subset J$,
$
  v\in C^\infty((0,\infty);C^1(K))
$.
Write now $v(t,y)=v(t)(y)$. Then, for every
$k\in\N_0$, the functions $\partial_t^k v$ and
$\partial_y\partial_t^k v$ are jointly continuous on $(0,\infty)\times J$. 

Set
next $
  u(t,x)=v(t,s(x))
$.
By \eqref{eq:L2-conjugation}, $u(t,\cdot)$ represents
$T_t^q f$ in $L^2(I,\mu)$. Since $u(t,\cdot)$ is continuous and $\mu$
has full support on $I$, this continuous representative is unique.
Since $s$ does not depend on $t$, for every $k\in\N_0$,
$
  \partial_t^ku(t,x)
  =
  \partial_t^kv(t,s(x))$.
Moreover, since $s'=\rho$,
\[
  \partial_x\partial_t^ku(t,x)
  =
  \partial_y\partial_t^kv(t,s(x))\rho(x).
\]
The joint continuity established above for
$\partial_t^kv$ and $\partial_y\partial_t^kv$, together with the
continuity of $s$ and $\rho$, proves \ref{item:main-a}.

For \ref{item:main-b}, \Cref{lem:domain} gives, for every
$t>0$,
$
  \partial_yv(t,\cdot)\in\AC_{\loc}(J)
$
and
\begin{align} \label{eq:260827.1}
  \partial_{yy}v(t,\cdot)
  =
  wA v(t,\cdot)
  =
  w\bigl(\partial_t v(t,\cdot)
  +(q\circ s^{-1})v(t,\cdot)\bigr)
\end{align}
almost everywhere.

On each compact subinterval of $I$, the map $s$ is bi-Lipschitz.
Consequently,
$
  \partial_yv(t,s(\cdot))\in\AC_{\loc}(I).
$
Since $\rho\in\AC_{\loc}(I)$, it follows that
\[
  \partial_xu(t,\cdot)=\partial_yv(t,s(\cdot))\rho\in\AC_{\loc}(I).
\]
Moreover, since $s^{-1}$ is locally Lipschitz, the preimage under $s$
of every Lebesgue-null subset of $J$ is a Lebesgue-null subset of $I$.
Thus the almost-everywhere identity in \eqref{eq:260827.1} remains valid
after composition with $s$.   The chain and product rules, together with
\eqref{eq:w}
and
$
  \rho'=-(b/a)\rho
$
almost everywhere, therefore give, for almost every $x$,
\begin{align}
  \partial_{xx}u(t,x)
  &=
  \partial_{yy}v(t,s(x))\rho^2(x)
  +\partial_yv(t,s(x))\rho'(x) \nonumber\\
  &=
  \frac{\partial_t v(t,s(x))+q(x)v(t,s(x))}{a(x)}
  -\frac{b(x)}{a(x)}\partial_xu(t,x) \nonumber\\
  &=
  \frac{\partial_tu(t,x)+q(x)u(t,x)
  -b(x)\partial_xu(t,x)}{a(x)}.  \label{eq:260821.2}
\end{align}
This is equivalent to \eqref{eq:a.e.-PDE}.

For every $k\in\N_0$, the first part of the proof gives
$
  \partial_t^k v(t,\cdot)\in\Dom(A)
$
for every $t>0$, and
$
  \partial_t^{k+1}v(t,\cdot)
  =
  B\partial_t^kv(t,\cdot)
$.
Applying the same argument with $v$ and $u$ replaced by
$\partial_t^kv$ and $\partial_t^ku$, respectively, shows that
$\partial_x\partial_t^ku(t,\cdot)\in\AC_{\loc}(I)$ and that
\eqref{eq:a.e.-PDE} holds almost everywhere with $u$ replaced by
$\partial_t^ku$.  This proves \ref{item:main-b}.

If $a$, $b$, and $q$ are continuous, the expression
for $\partial_{xx}u$ in \eqref{eq:260821.2} is then jointly continuous in $(t,x)$.  Since
$\partial_xu(t,\cdot)$ is locally absolutely continuous and has this function as
its almost-everywhere derivative, the fundamental theorem of calculus shows
that $\partial_{xx}u$ exists everywhere and is jointly continuous.  Hence
$
  u\in C^{1,2}((0,\infty)\times I)$,
and the equation holds pointwise on $(0,\infty)\times I$.  Applying the same argument to
$B^kv=\partial_t^kv$ proves the assertion for every time derivative and
hence \ref{item:main-c}.

 The final assertion follows from the strong
continuity of $(T_t^q)_{t\ge0}$.
\end{proof}

\begin{remark}\label{rem:Ekstrom-Janson-Tysk}
When $I=\R$, $b=q=0$, $\sigma$ is continuous and nonzero, and the
initial datum is continuous and satisfies their global growth
condition, the positive-time $C^{1,2}$ conclusion of
\Cref{thm:main}\ref{item:main-c} follows from
\cite[Theorem~8.1 and Corollary~8.3]{EkstromJansonTysk2015}.
Compared with that result, \Cref{thm:main} allows arbitrary
$L^2(I,\mu)$ initial data, general Engelbert--Schmidt coefficients
before passage to scale, bounded Borel potentials, arbitrary intervals,
and arbitrary admissible form domains.
\end{remark}

Theorem~\ref{thm:main}  is formulated for an arbitrary admissible form domain
$\cV$.  We now record the boundary behavior for the two canonical
choices $\cV_0$ and $\cV_{\max}$.  The domain $\cV_0$ gives the
homogeneous Dirichlet condition at finite scale endpoints.  For
$\cV_{\max}$, a finite speed measure near an endpoint yields a vanishing
scale derivative; if $\rho$
is bounded there, the derivative in the original coordinate also
vanishes.  At endpoints of finite scale with finite speed measure,  this is the classical Neumann-type boundary behavior
for the associated symmetric reflected diffusion; see
\cite{FukushimaOshimaTakeda2010,Fukushima2010,Fukushima2014}.

\begin{proposition}[Boundary behavior for canonical form domains]\label{prop:boundary-realizations}
Let $u$ denote the solution from \Cref{thm:main}. The following boundary conditions hold.

\begin{enumerate}[label=\textup{(\roman*)},ref=\textup{(\roman*)}]
\item\label{item:boundary-dirichlet}
For $\cV=\cV_0$,
\[
  \lim_{x\downarrow\ell}u(t,x)=0
  \quad\text{if }s(\ell+)>-\infty,
  \qquad
  \lim_{x\uparrow r}u(t,x)=0
  \quad\text{if }s(r-)<\infty,
\]
for every $t>0$.

\item\label{item:boundary-neumann}
For $\cV=\cV_{\max}$, let $e\in\{\ell,r\}$ and suppose that $\mu$ has
finite mass near $e$. Then, for every $t>0$,
\[
  \lim_{x\to e}
  \frac{\partial_xu(t,x)}{\rho(x)}
  =
  0,
\]
where the limit is taken within $I$. If, in addition, $\rho$ is bounded
near $e$, then
$
  \lim_{x\to e}\partial_xu(t,x)=0.
$
\end{enumerate}
\end{proposition}

\begin{proof}
Fix $f\in L^2(I,\mu)$ and $t>0$, and set
$
  v_t=e^{tB}Uf
$.
By \Cref{lem:semigroup-domain},
$
  v_t\in\Dom(B)=\Dom(A)
$,
where $A$ is the operator associated with the form domain under
consideration.  Moreover, by the proof of \Cref{thm:main},
$
  u(t,x)=v_t(s(x))
$.

For \ref{item:boundary-dirichlet}, if $\cV=\cV_0$, then
$
  v_t\in\Dom(A)\subset\cV_0
$.
Hence
$
  v_t(y)\to0
$
as $y\downarrow\alpha$ when $\alpha>-\infty$, and as
$y\uparrow\beta$ when $\beta<\infty$.  Since
$s(x)\downarrow\alpha$ as $x\downarrow\ell$ and
$s(x)\uparrow\beta$ as $x\uparrow r$, the asserted boundary limits for
$u$ follow.

For \ref{item:boundary-neumann}, consider the left endpoint, i.e.~$e=\ell$, and suppose
that
$
  \mu((\ell,c))<\infty
$
for some $c\in I$.  Equivalently,
$
  \int_\alpha^{s(c)} w(y)\,\dd y<\infty$.
 \Cref{lem:domain} gives
$
  v_t''=wAv_t
$
almost everywhere. Hence
\[
  \int_\alpha^{s(c)} |v_t''(y)|\,\dd y
  \le
  \|Av_t\|_H
  \left(\int_\alpha^{s(c)} w(y)\,\dd y\right)^{1/2}
  <\infty.
\]
Thus $v_t'(y)$ has a finite limit as $y\downarrow\alpha$.

If $\alpha=-\infty$, this limit must be zero because
$
  v_t'\in L^2(J)
$.
If $\alpha>-\infty$, choose $z\in\cV_{\max}$ such that $z=1$ near
$\alpha$ and $z=0$ near $s(c)$.  Using
\eqref{eq:generator-form} and $v_t''=wAv_t$, integration by parts
gives
\[
  -\int_\alpha^{s(c)}v_t''(y)z(y)\,\dd y
  =
  \int_\alpha^{s(c)}v_t'(y)z'(y)\,\dd y
  =
  -\lim_{y\downarrow\alpha}v_t'(y)
  -\int_\alpha^{s(c)}v_t''(y)z(y)\,\dd y.
\]
Therefore
\[
 \lim_{x\downarrow\ell}\frac{\partial_xu(t,x)}{\rho(x)}=
  \lim_{x\downarrow\ell}v_t'(s(x)) = \lim_{y\downarrow\alpha}v_t'(y)=0.
\]
The argument at the right endpoint is analogous.  This proves
\ref{item:boundary-neumann}.
\end{proof}

We next turn to nonzero, time-independent Dirichlet data, which are handled by subtracting a function
that is affine in the scale variable. To this end, let $J=s(I)=(\alpha,\beta)$.  Prescribe
$\gamma_\ell\in\mathbb R$ if $\alpha>-\infty$ and
$\gamma_r\in\mathbb R$ if $\beta<\infty$, and define
$\xi:I\to\mathbb R$ by
\begin{align} \label{eq:260824.4}
  \xi(x)
  =
  \begin{cases}
    \displaystyle
    \gamma_\ell
    +(\gamma_r-\gamma_\ell)
      \frac{s(x)-\alpha}{\beta-\alpha},
      & \alpha>-\infty,\ \beta<\infty,\\[2ex]
    \gamma_\ell,
      & \alpha>-\infty,\ \beta=\infty,\\
    \gamma_r,
      & \alpha=-\infty,\ \beta<\infty,\\
    0,
      & \alpha=-\infty,\ \beta=\infty.
  \end{cases}
\end{align}
By construction, 
$\xi\in C^1(I)$, $\xi'\in\AC_{\loc}(I)$, and, since
$s'=\rho$ and $\rho'=-(b/a)\rho$ almost everywhere,
\[
  a\xi''+b\xi'=0
  \qquad\text{almost everywhere on }I.
\]

\begin{proposition}[Nonzero, time-independent Dirichlet data]
\label{prop:nonzero-dirichlet}
Assume \eqref{eq:ES}, let $q:I\to\mathbb R$ be bounded and Borel, and
let $f:I\to\mathbb R$ be Borel.  Suppose that
\begin{equation*}
  f-\xi\in L^2(I,\mu),
  \qquad
  q\xi\in L^2(I,\mu).
\end{equation*}
Then there exists a function
$
  u:(0,\infty)\times I\to\mathbb R
$
satisfying
\Cref{thm:main}\ref{item:main-a}--\ref{item:main-c}.
Moreover, for every $t>0$,
$
  u(t,\cdot)-\xi\in L^2(I,\mu)
$,
\begin{align} \label{eq:bus}
  \lim_{x\downarrow\ell}u(t,x)=\gamma_\ell
  \quad\text{if }s(\ell+)>-\infty,
  \qquad
  \lim_{x\uparrow r}u(t,x)=\gamma_r
  \quad\text{if }s(r-)<\infty,
\end{align}
and
$
  u(t,\cdot)- f \to 0
$
in $L^2(I,\mu)$ as $t\downarrow0$.
\end{proposition}

\begin{proof}
For $\cV=\cV_0$, let $A$ and $B$ be as in
\Cref{prop:L2-semigroup}. Set
$
  g=U[f-\xi]\in H
$
and
$
  h=U[q\xi]\in H
$,
and define, for $t>0$,
\[
  z(t)
  =
  e^{tB}g-\int_0^t e^{rB}h\,\dd r,
  \qquad
  u(t,x)=\xi(x)+z(t,s(x)).
\]
By \Cref{lem:semigroup-domain},
$
  z\in C^\infty((0,\infty);\Dom(B))
$
and
$
  z'=Bz-h
$.
Since $\Dom(B)=\Dom(A)$ and the graph norms of $A$ and $B$ are
equivalent, the argument in the proof of \Cref{thm:main}, with $z$ in
place of $v$, gives \ref{item:main-a} and
$
  \partial_xu(t,\cdot)\in\AC_{\loc}(I)
$
for every $t>0$.

Moreover, for every $t>0$ and almost every $x\in I$,
\[
  \partial_t z(t,s(x))
  =
  a(x)\partial_{xx}z(t,s(x))
  +b(x)\partial_x z(t,s(x))
  -q(x)z(t,s(x))-q(x)\xi(x).
\]
Since
$
  a\xi''+b\xi'=0
$
almost everywhere on $I$, it follows that
\eqref{eq:a.e.-PDE} holds for $u$. For every $k\in\N$,
differentiating $z'=Bz-h$ in time gives
$
  \partial_t^{k+1}z=B\partial_t^kz
$.
The argument in the proof of \Cref{thm:main}\ref{item:main-b} therefore
gives the remaining assertions in \ref{item:main-b}. If $a$, $b$, and
$q$ are continuous on $I$, then $\xi\in C^2(I)$, and the argument in
the proof of \Cref{thm:main}\ref{item:main-c} gives
\ref{item:main-c}.

Since $z(t)\in H$,
$
  u(t,\cdot)-\xi\in L^2(I,\mu)
$
for every $t>0$. Moreover,
$
  z(t,\cdot)\in\Dom(A)\subset\cV_0
$,
so
$
  z(t,y)\to0
$
as $y\downarrow s(\ell+)$ whenever $s(\ell+)>-\infty$, and
$
  z(t,y)\to0
$
as $y\uparrow s(r-)$ whenever $s(r-)<\infty$.
Together with the boundary values of $\xi$, this gives \eqref{eq:bus}.
Finally,
\[
  \|z(t)-g\|_H
  \le
  \|e^{tB}g-g\|_H
  +
  t e^{t\|q^-\|_\infty}\|h\|_H
  \longrightarrow0
\]
as $t\downarrow0$. Since
$
  U[u(t,\cdot)-\xi]=z(t)
$
and
$
  U[f-\xi]=g
$,
the unitarity of $U$ yields
$
  u(t,\cdot)-f\to0
$
in $L^2(I,\mu)$ as $t\downarrow0$.
\end{proof}

\section{Feynman--Kac representations}
\label{sec:FK}

In this section, we identify the semigroups constructed in
\Cref{sec:main} with their Feynman--Kac representations.  In
\Cref{sec:FK-dirichlet}, we treat killed diffusion and nonzero,
time-independent Dirichlet data.  In \Cref{sec:FK-reflected}, we treat the
reflected diffusion associated with the maximal form domain.
 The regularity and
boundary properties of these probabilistic representations then follow
from the results of \Cref{sec:main}.

\subsection{Killed diffusion and Dirichlet data}
\label{sec:FK-dirichlet}

Throughout this subsection, take $\cV=\cV_0$, let
$q:I\to\mathbb R$ be bounded and Borel, and let $A$, $B$, and
$(T_t^q)_{t\ge0}$ be the corresponding objects from \Cref{sec:main}.  By
\cite[Theorem~3.2]{Fukushima2010}, $(\cE,\cV_0)$ is a regular,
strongly local Dirichlet form on $H$, and by
\cite[Theorem~3.3]{Fukushima2010} its associated diffusion is the natural-scale diffusion $Y$ killed on
leaving $J$.  In particular,
$
  e^{tA}
$
is its $L^2$ transition semigroup.

Write $\E_x$ for expectation with respect to the law of $X$ started
from $x\in I$. For $t>0$ and $f\in L^2(I,\mu)$, choose a Borel
representative $\widehat f$ of $f$ and set
\begin{equation}\label{eq:FK-killed}
  S_t^qf(x)
  =
  \E_x\!\left[
    \exp\!\left(-\int_0^t q(X_r)\,\dd r\right)
    \widehat f(X_t);
    \ t<\zeta
  \right],
  \qquad x\in I.
\end{equation}
The following proposition shows that the definition is independent of
the chosen representative and identifies $S_t^q$ with $T_t^q$.

\begin{proposition}[Feynman--Kac identification]
\label{prop:FK-L2}
Assume \eqref{eq:ES}, and let $q:I\to\mathbb R$ be bounded and Borel.
For every $t>0$ and every $f\in L^2(I,\mu)$, the expectation in
\eqref{eq:FK-killed} is absolutely integrable for every $x\in I$ and
is independent of the chosen Borel representative $\widehat f$.
Moreover, $S_t^qf$ is the unique continuous representative of
$T_t^qf$.
\end{proposition}

\begin{proof}
Fix $t>0$ and $f\in L^2(I,\mu)$, and let $\widehat f$ be a Borel
representative of $f$. Let $p_t(x,z)$ denote the transition density of
the killed diffusion with respect to $\mu$. It can be chosen symmetric
and jointly continuous on $(0,\infty)\times I\times I$; see
\cite[Chapter~IV, Section~4.11]{ItoMcKean1974}. By symmetry and the
Chapman--Kolmogorov identity,
\[
  \int_I p_t(x,z)^2\,\mu(\dd z)
  =
  \int_I p_t(x,z)p_t(z,x)\,\mu(\dd z)
  =
  p_{2t}(x,x)
  <\infty
\]
for every $x\in I$. Hence
\begin{align}
  \E_x\!\left[
    \exp\!\left(-\int_0^t q(X_r)\,\dd r\right)
    |\widehat f(X_t)|;
    \ t<\zeta
  \right]
  &\le
  e^{t\|q^-\|_\infty}
  \int_I p_t(x,z)|\widehat f(z)|\,\mu(\dd z) \nonumber\\
  &\le
  e^{t\|q^-\|_\infty}
  p_{2t}(x,x)^{1/2}
  \|f\|_{L^2(I,\mu)}.
  \label{eq:260824.2}
\end{align}
This proves absolute integrability. If $\widetilde f$ is another
Borel representative of $f$, applying \eqref{eq:260824.2} with
$\widehat f-\widetilde f$ in place of $\widehat f$ shows that $S_t^qf$ does not depend on this choice.

Write now $\E_y^Y$ for expectation with respect to the law of $Y$ started
from $y\in J$. Set
$
  c=\|q^-\|_\infty
$
and
$
  R=q\circ s^{-1}+c\ge0
$.
The positive continuous additive functional
\[
  C_t^R
  =
  \int_0^{t\wedge\zeta}R(Y_r)\,\dd r
\]
has Revuz measure
$
  \nu(\dd y)=R(y)w(y)\,\dd y
$.
Indeed, for nonnegative bounded Borel functions $\varphi$ and $\psi$
with compact support, symmetry of $(e^{tA})_{t\ge0}$ gives
\[
  \int_J \psi(y)\E_y^Y\!\left[
    \int_0^t \varphi(Y_r)\,\dd C_r^R
  \right]w(y)\,\dd y
  =
  \int_0^t\int_J
    e^{rA}\psi(y)\varphi(y)R(y)w(y)\,\dd y\,\dd r.
\]
After extending this identity to arbitrary nonnegative Borel
$\varphi$ and $\psi$ by monotone convergence, the Revuz
characterization
\cite[Theorem~5.1.3 (i)$\Leftrightarrow$(iii)]
{FukushimaOshimaTakeda2010}
yields the claim.

Since $R$ is bounded, \cite[Theorem~6.1.1]{FukushimaOshimaTakeda2010} shows that the process obtained
from $Y$ by additional killing with $C^R$ has Dirichlet form
$(\cE^\nu,\cV_0)$, where
\[
  \cE^\nu(v,z)
  =
  \cE(v,z)+\int_J vz\,\dd\nu.
\]
Let $R$ also denote the multiplication operator by $R$ on $H$. The
operator associated with $(\cE^\nu,\cV_0)$ is $A-R$ with domain
$\Dom(A)$. Indeed, for $v\in\Dom(A)$ and $z\in\cV_0$,
\[
  \cE^\nu(v,z)
  =
  -\langle Av,z\rangle_H+\langle Rv,z\rangle_H
  =
  -\langle(A-R)v,z\rangle_H.
\]
Conversely, if $v$ belongs to the operator domain associated with
$\cE^\nu$, then for some $g\in H$,
$
  \cE(v,z)=-\langle g+Rv,z\rangle_H
$
for every $z\in\cV_0$, hence $v\in\Dom(A)$.

It remains to identify the semigroup.  By linearity, it suffices to
consider nonnegative $f$. Choose $\widehat f\ge0$ and set
$
  h=\widehat f\circ s^{-1}
$,
so that $[h]=Uf$. By
\cite[Theorem~4.2.3(i)]{FukushimaOshimaTakeda2010}, the function
\begin{equation}\label{eq:260824.1}
  y\longmapsto
  \E_y^Y\!\left[
    \exp\!\left(-\int_0^tR(Y_r)\,\dd r\right)
    h(Y_t);
    \ t<\zeta
  \right]
\end{equation}
is an $(\cE^\nu,\cV_0)$-quasi-continuous representative of
$e^{t(A-R)}[h]$. By
\cite[Example~2.1.2, p.~77]{FukushimaOshimaTakeda2010},
quasi-continuity here is indeed ordinary continuity. Hence the function in
\eqref{eq:260824.1} is the continuous representative of
$e^{t(A-R)}[h]$.

Since $R=q\circ s^{-1}+c$ and $B=(A-R)+c$, multiplying the function in \eqref{eq:260824.1} by $e^{ct}$ and using
$Y=s(X)$ on $[0,\zeta)$ shows that $S_t^q f\circ s^{-1}$ is the
continuous representative of $e^{tB}Uf$. By \eqref{eq:L2-conjugation}, $S_t^qf$ is therefore a continuous
representative of $T_t^qf$.
Its uniqueness follows from \Cref{thm:main}.
\end{proof}

By \Cref{prop:FK-L2} and \Cref{thm:main}, the function
$
  u(t,x)=S_t^qf(x)
$
satisfies all the conclusions of \Cref{thm:main} for $\cV=\cV_0$.
Moreover, \Cref{prop:boundary-realizations}\ref{item:boundary-dirichlet}
gives
$
  \lim_{x\downarrow\ell}u(t,x)=0
$
if $s(\ell+)>-\infty$, and
$
  \lim_{x\uparrow r}u(t,x)=0
$
if $s(r-)<\infty$, for every $t>0$.

For the nonzero, time-independent Dirichlet data of
\Cref{prop:nonzero-dirichlet}, retain $\gamma_\ell$, $\gamma_r$, and
the function $\xi$ from \eqref{eq:260824.4}.  The function $\xi$ is
bounded and has limits at both endpoints of $I$.  On
$\{\zeta<\infty\}$, write
$
  X_{\zeta-}
  =
  \lim_{r\uparrow\zeta}X_r
$
and interpret $\xi(X_{\zeta-})$ through these endpoint limits.

\begin{corollary}[Feynman--Kac representation with nonzero, time-independent Dirichlet data]
\label{cor:FK-nonzero-dirichlet}
Under the assumptions of \Cref{prop:nonzero-dirichlet}, for $t>0$ and
$x\in I$, define
\begin{align}\label{eq:FK-nonzero-dirichlet}
  F(t,x)
  &=
  \E_x\!\left[
    \exp\!\left(-\int_0^t q(X_r)\,\dd r\right)
    f(X_t);
    \ t<\zeta
  \right] \nonumber\\
  &\quad+
  \E_x\!\left[
    \exp\!\left(-\int_0^\zeta q(X_r)\,\dd r\right)
    \xi(X_{\zeta-});
    \ \zeta\le t
  \right].
\end{align}
Both expectations are absolutely integrable, and, if $u$ denotes the
solution constructed in \Cref{prop:nonzero-dirichlet}, then
$
  F(t,x)=u(t,x)$ for all $t >0$ and $x\in I$.
\end{corollary}

\begin{proof}
Since
$
  f=(f-\xi)+\xi
$,
\eqref{eq:260824.2}, applied to $f-\xi$, together with the boundedness
of $\xi$, shows that the first expectation in
\eqref{eq:FK-nonzero-dirichlet} is absolutely integrable. The second
expectation is absolutely integrable because, on $\{\zeta\le t\}$, its
integrand has absolute value at most
$
  e^{t\|q^-\|_\infty}\|\xi\|_\infty
$.

By the proof of \Cref{prop:nonzero-dirichlet} and the unitarity of $U$,
\[
  [u(t,\cdot)-\xi]
  =
  T_t^q[f-\xi]
  -
  \int_0^t T_r^q[q\xi]\,\dd r.
\]
By \Cref{prop:FK-L2}, dominated convergence, \eqref{eq:FK-L2-bound},
and Fubini's theorem, the continuous representative of the right-hand
side is
\begin{align*}
  &\E_x\!\left[
    \exp\!\left(-\int_0^t q(X_v)\,\dd v\right)
    (f-\xi)(X_t);
    \ t<\zeta
  \right] \\
  &\quad-
  \int_0^t
  \E_x\!\left[
    \exp\!\left(-\int_0^r q(X_v)\,\dd v\right)
    q(X_r)\xi(X_r);
    \ r<\zeta
  \right]\dd r.
\end{align*}
Since $u(t,\cdot)-\xi$ is continuous and $\mu$ has full support, this
equals $u(t,x)-\xi(x)$ for every $x\in I$.

Let $(\tau_n)_{n\in\N}$ be exit times from an increasing sequence of
compact subintervals of $I$ with $\tau_n\uparrow\zeta$. Since
$\xi\in C^1(I)$, $\xi'\in\AC_{\loc}(I)$, and
$
  a\xi''+b\xi'=0
$
almost everywhere, the generalized It\^o formula
\cite[Theorem~3.7.1]{KaratzasShreve1991}, applied to
$\xi(X_{\cdot\wedge\tau_n})$, and the product rule give, after letting
$n\uparrow\infty$,
\begin{align*}
  \xi(x)
  &=
  \E_x\!\left[
    \exp\!\left(-\int_0^t q(X_v)\,\dd v\right)
    \xi(X_t);
    \ t<\zeta
  \right] \\
  &\quad+
  \E_x\!\left[
    \exp\!\left(-\int_0^\zeta q(X_v)\,\dd v\right)
    \xi(X_{\zeta-});
    \ \zeta\le t
  \right] \\
  &\quad+
  \int_0^t
  \E_x\!\left[
    \exp\!\left(-\int_0^r q(X_v)\,\dd v\right)
    q(X_r)\xi(X_r);
    \ r<\zeta
  \right]\dd r,
\end{align*}
where dominated convergence and Fubini's theorem were used. Adding
the last two displays gives
$
  F(t,x)=u(t,x)
$.
\end{proof}

\begin{remark}[Classical verification]
The preceding results identify the Feynman--Kac values with the
solutions in \Cref{thm:main} and \Cref{prop:nonzero-dirichlet}, but do
not imply uniqueness among all classical solutions.  Let $v$ be a
classical solution with the same initial and boundary data.  For the
homogeneous Dirichlet problem, apply It\^o's formula to
$
  r\mapsto
  \exp\!\left(-\int_0^r q(X_w)\,\dd w\right)v(t-r,X_r)
$
up to the exit times from compact subintervals of $I$.  If the resulting
terminal variables are uniformly integrable, passage to the limit gives
$
  v(t,x)=S_t^qf(x)
$;
compare \cite[Theorem~2.5]{JansonTysk2006}.  With nonzero, time-independent 
Dirichlet data, the same argument gives the exit payoff in
\eqref{eq:FK-nonzero-dirichlet} and identifies $v$ with $F$.

If $I$ is bounded and $v$ extends continuously to
$[0,T]\times\overline I$, the localized terminal variables are bounded,
so uniform integrability is automatic. On an unbounded interval,
additional integrability may be needed to pass to the localization
limit. Such a condition cannot in general be omitted: strict local
martingales provide the standard obstruction and may lead to
nonuniqueness of classical solutions; see
\cite{CoxHobson2005,EkstromTysk2009,BayraktarXing2010,Ruf2013}.
For uniqueness of weak solutions to related Cauchy problems under
Engelbert--Schmidt conditions, see also \cite{CetinLarsen2023}.
\end{remark}

\subsection{Maximal form domain and reflected diffusion}
\label{sec:FK-reflected}

Throughout this subsection, take $\cV=\cV_{\max}$, let
$q:I\to\mathbb R$ be bounded and Borel, and let $A$, $B$, and
$(T_t^q)_{t\ge0}$ be the corresponding objects from \Cref{sec:main}.

Call the left endpoint regular if
\[
  s(\ell+)>-\infty
  \qquad\text{and}\qquad
  \mu((\ell,c))<\infty
\]
for some $c\in I$, and the right endpoint regular if
\[
  s(r-)<\infty
  \qquad\text{and}\qquad
  \mu((c,r))<\infty
\]
for some $c\in I$.  Let $I^{\mathrm r}$ be obtained from $I$ by
adjoining exactly the regular endpoints, and let $J^{\mathrm r}$ be
the corresponding extension of $J$.  The scale
function extends to a homeomorphism from $I^{\mathrm r}$ onto
$J^{\mathrm r}$, again denoted by $s$.  Extend $w(y)\,\dd y$ by zero mass at the added endpoints, and denote the resulting
measure by $m^{\mathrm r}$.  It has full support on
$J^{\mathrm r}$, and
$
  L^2(J^{\mathrm r},m^{\mathrm r})
$
is naturally identified with $H$.

By \cite[Theorem~5.2 and the discussion following (5.6)]
{Fukushima2010} and
\cite[Theorem~2.2(i), (iii), and Section~2.3]
{Fukushima2014}, $(\cE,\cV_{\max})$, regarded as a form on
$J^{\mathrm r}$, is a regular, strongly local, irreducible Dirichlet
form.  Let $Y^{\mathrm r}$ be its associated diffusion.  Killing
$Y^{\mathrm r}$ on leaving $J$ gives $Y$; at the added endpoints,
$Y^{\mathrm r}$ is instantaneously reflecting: it is not killed and spends zero Lebesgue time there.
Pulling $Y^{\mathrm r}$ back through the scale function gives a
diffusion $X^{\mathrm r}$ on $I^{\mathrm r}$, whose killing on leaving
$I$ gives $X$.  We call $X^{\mathrm r}$ the reflected diffusion and
write $\zeta^{\mathrm r}$ for its lifetime, which may still be
finite if $X^{\mathrm r}$ reaches a non-adjoined endpoint in finite
time.  Write $\E_x^{\mathrm r}$ for expectation under its law when
started from $x\in I^{\mathrm r}$.

For $t>0$ and $f\in L^2(I,\mu)$, choose a Borel representative
$\widehat f$ of $f$. Extend $q$ and $\widehat f$ to a bounded Borel function
$q^{\mathrm r}$ on $I^{\mathrm r}$ and to a Borel
function $\widehat f^{\mathrm r}$ on $I^{\mathrm r}$, respectively, by assigning
arbitrary finite values at the added endpoints. Set
\begin{equation}\label{eq:FK-reflected}
  S_t^{q,\mathrm r}f(x)
  =
  \E_x^{\mathrm r}\!\left[
    \exp\!\left(
      -\int_0^t q^{\mathrm r}(X_v^{\mathrm r})\,\dd v
    \right)
    \widehat f^{\mathrm r}(X_t^{\mathrm r});
    \ t<\zeta^{\mathrm r}
  \right],
  \qquad x\in I^{\mathrm r}.
\end{equation}
The next proposition shows that neither extension nor the chosen Borel
representative affects this definition.

\begin{proposition}[Feynman--Kac representation for the reflected diffusion]
\label{prop:FK-max}
Assume \eqref{eq:ES}, and let $q:I\to\mathbb R$ be bounded and Borel.
For every $t>0$ and every $f\in L^2(I,\mu)$, the expectation in
\eqref{eq:FK-reflected} is absolutely integrable for every
$x\in I^{\mathrm r}$ and is independent of the choices of
$\widehat f$, $q^{\mathrm r}$, and $\widehat f^{\mathrm r}$.
Moreover,
$
  S_t^{q,\mathrm r}f\in C(I^{\mathrm r})
$,
and its restriction to $I$ is the unique continuous representative of $T_t^qf$.
\end{proposition}

\begin{proof}
Since $X^{\mathrm r}$ spends zero Lebesgue time at the added endpoints, the time integral in \eqref{eq:FK-reflected} is independent of the values assigned to $q^{\mathrm r}$ there.
By \cite[Chapter~IV, Section~4.11]{ItoMcKean1974}, $Y^{\mathrm r}$ has a symmetric, jointly continuous transition density $p_t^{\mathrm r}(y,z)$ with respect to $m^{\mathrm r}$. The transition-density estimate used in the proof of \Cref{prop:FK-L2}, with $p_t$ and $\mu$ replaced by $p_t^{\mathrm r}$ and $m^{\mathrm r}$, respectively, proves absolute integrability in \eqref{eq:FK-reflected}. It also proves independence of $\widehat f$ and $\widehat f^{\mathrm r}$, since any two resulting functions on $J^{\mathrm r}$ agree $m^{\mathrm r}$-almost everywhere. 

The continuity assertion and identification with $T_t^qf$ follow by the same arguments as in the proof of \Cref{prop:FK-L2}, with the Dirichlet realization replaced by the maximal realization.
\end{proof}

By \Cref{prop:FK-max} and \Cref{thm:main}, for every
$f\in L^2(I,\mu)$ the function
$
  u(t,x)=S_t^{q,\mathrm r}f(x)
$,
$t>0$ and $x\in I$, satisfies all the conclusions of
\Cref{thm:main} for $\cV=\cV_{\max}$. 
Proposition~\ref{prop:boundary-realizations}\ref{item:boundary-neumann}
gives, at every regular endpoint $e\in\{\ell,r\}$,
\[
  \lim_{x\to e}
  \frac{\partial_xu(t,x)}{\rho(x)}
  =
  0,
  \qquad t>0,
\]
where the limit is taken within $I$. This is the homogeneous Neumann
condition in the scale coordinate. If $\rho$ is bounded near $e$, then
also
$
  \lim_{x\to e}\partial_xu(t,x)=0.
$

\section{Sharpness and a two-dimensional counterexample}
\label{sec:sharpness}
The three cases of \Cref{ex:jump-coefficients} lie in the
Engelbert--Schmidt framework and show that discontinuity of any one of
$a$, $b$, or $q$ can destroy twice differentiability.  In that example,
$A$ and $B$ denote the operators associated with the Dirichlet form
domain $\cV_0$.  By \Cref{prop:FK-L2}, the corresponding semigroup
values are the killed Feynman--Kac values.  \Cref{ex:260824} shows that
the one-dimensional regularity conclusion does not extend to two
dimensions under continuity and uniform ellipticity alone.

\begin{example}[Discontinuous coefficients]\label{ex:jump-coefficients}
Consider the following three choices of $I$, $a$, $b$, $q$, and $F$.
\begin{enumerate}[label=\textup{(\roman*)},ref=\textup{(\roman*)}]
\item\label{item:jump-a}
Let
\[
  I=\left(-\frac\pi2,\frac\pi{\sqrt2}\right),
  \qquad b=q=0,
  \qquad
  a(x)=
  \begin{cases}
    1,&x<0,\\
    2,&x\ge0,
  \end{cases}
\]
and define
\[
  F(x)=
  \begin{cases}
    \cos x,&x\le0,\\
    \cos(x/\sqrt2),&x\ge0.
  \end{cases}
\]

\item\label{item:jump-b}
Let
\[
  I=\left(-\frac{3\pi}{4},\frac{2\pi}{3\sqrt3}\right),
  \qquad
  a=1,\qquad q=0,\qquad
  b(x)=\one_{[0,\infty)}(x),
\]
and define
\[
  F(x)=
  \begin{cases}
    \cos x-\sin x,&x\le0,\\[1mm]
    e^{-x/2}\left(
      \cos\frac{\sqrt3x}{2}
      -\frac1{\sqrt3}\sin\frac{\sqrt3x}{2}
    \right),&x\ge0.
  \end{cases}
\]

\item\label{item:jump-q}
Let
\[
  I=\left(-\frac{3\pi}{4},1\right),
  \qquad a=1,\qquad b=0,\qquad
  q(x)=\one_{[0,\infty)}(x),
\]
and define
\[
  F(x)=
  \begin{cases}
    \cos x-\sin x,&x\le0,\\
    1-x,&x\ge0.
  \end{cases}
\]
\end{enumerate}

In every case, the Engelbert--Schmidt conditions hold with
$\sigma=\sqrt{2a}$.  Moreover, $F \in C^1(I)$, $F$
vanishes at both endpoints, $F'\in\AC(I)$, and
$
  aF''+bF'-qF=-F
$
almost everywhere.  On the other hand, $F''(0-)=-1$ in every case,
whereas $F''(0+)=-1/2$ in \ref{item:jump-a} and $F''(0+)=0$ in
\ref{item:jump-b} and \ref{item:jump-q}.  Thus $F$ is not twice
differentiable at $0$.

In each case, set $G=UF=F\circ s^{-1}$.  Since $I$ is bounded, $a$, $b$, and $1/a$ are bounded, $\rho$ is bounded above
and away from zero.  Since $F$ and $F'$ are bounded and $F$ vanishes
at the endpoints, a change of variables gives
$
  G\in\cV_0
$.
Moreover, $G'=
  ({F'}/{\rho})\circ s^{-1}
  \in \AC(J)$, and the chain rule, together with
$\rho'=-(b/a)\rho$ and \eqref{eq:w}, gives
\[
  G''(s(x))
  =
  \frac{F''(x)\rho(x)-F'(x)\rho'(x)}{\rho^3(x)}
  =
  \frac{a(x)F''(x)+b(x)F'(x)}{a(x)\rho^2(x)}
  =
  w(s(x))(q(x)-1)F(x)
\]
almost everywhere.  Since $q$ is bounded, $(Q-I)G\in H$ and integration by parts in
\eqref{eq:generator-form} yields
$
  G\in\Dom(A)
$
and
$
  AG=QG-G
$.
Hence $G\in\Dom(B)$ and $BG=-G$.  Consequently,
\[
  T_t^qF
  =
  U^{-1}e^{tB}G
  =
  e^{-t}F,
  \qquad t\ge0.
\]
Thus $T_t^qF$ is not twice differentiable at $0$ for any $t>0$.
Cases \ref{item:jump-a}, \ref{item:jump-b}, and \ref{item:jump-q}
show, respectively, that continuity of $a$, $b$, and $q$ cannot be
omitted from \Cref{thm:main}\ref{item:main-c}, even when the other two
coefficients are continuous.
\end{example}

\begin{example}[Failure of $C^{1,2}$
 regularity in two dimensions]
\label{ex:260824}
The one-dimensional conclusion of \Cref{thm:main} does not extend to
two dimensions under mere continuity of the coefficient matrix.  We use
the construction in the proof of
\cite[Theorem~4]{EscauriazaMontaner2017}.  For $r>0$ set 
\[
  L(r)=4-\log r,
  \qquad
  \alpha(r)
  =
  \frac{4L(r)-1}{L(r)^2-3L(r)+1},
\]
and
define
\[
  M(x)
  =
  I_2+\one_{x \neq 0} \frac{\alpha(|x| \wedge 1)}{|x|^2}
  \begin{pmatrix}
    x_1^2 & x_1x_2\\
    x_1x_2 & x_2^2
  \end{pmatrix}, 
\]
with $I_2$ denoting the $2\times 2$ identity matrix.
The matrix $M$ is continuous and uniformly elliptic, but is not
H\"older continuous at the origin (this is already seen along the
$x_1$-axis).

Define
\[
  \Phi(x)
  =
 \one_{x \neq 0}  x_1x_2
  \left(4-\log|x|\right)^2.
\]
For $0<|x|<1$, a direct calculation in polar coordinates gives
\[
\operatorname{tr}\!\left(M(x)D^2\Phi(x)\right)
=
\frac{2x_1x_2}{|x|^2}
\left(
1-4L(|x|)
+\alpha(|x|)\bigl(L(|x|)^2-3L(|x|)+1\bigr)
\right)
=0.
\]
Thus, $\operatorname{tr}(M D^2\Phi)=0$ almost everywhere on the unit disc
$B_1 =\{x\in\R^2:|x|<1\}$.
The function $\Phi$, its first
derivatives, and its almost-everywhere defined second derivatives
belong to $L^p(B_1)$ for every $p<\infty$.  Its mixed second derivative
is nevertheless unbounded in every neighbourhood of the origin; in particular,
$
  \Phi\notin C^2(B_1).
$

We next turn to the parabolic Feynman--Kac interpretation.  Since $M$ is
bounded, continuous, and uniformly elliptic, the martingale problem for
$
  f\mapsto\operatorname{tr}(M D^2f)
$
is well posed; see \cite[Theorem~7.2.1]{StroockVaradhan1979}.  Let $Z$
be the corresponding diffusion, equivalently a weak solution of
\[
  \dd Z_t=\sqrt{2M(Z_t)}\,\dd W_t,
\]
where $W$ now denotes a two-dimensional Brownian motion and $\sqrt{2M(Z_t)}$ is the symmetric positive-definite square root of $2 M(Z_t)$. 
Let
$
  \tau=\inf\{t\ge0:Z_t\notin B_1\}.
$
By the It\^o--Krylov formula \cite[Theorem~2.10.1]{Krylov1980}, with $\E_x^Z$ denoting the expectation for the process $Z$ started at $x$, 
\[
  \Phi(x)
  =
  \E_x^Z\!\left[\Phi(Z_{t\wedge\tau})\right],
  \qquad t\ge0,\quad x\in B_1.
\]
Thus the Feynman--Kac value with initial condition $\Phi$ and boundary values
 $\Phi|_{\partial B_1}$ is stationary, i.e., 
$
  u(t,x)
  =
  \Phi(x)$.
In particular, it is not twice differentiable at the origin for any
$t>0$.

Consequently, the parabolic initial-boundary value problem
\[
  \partial_tu
  =
  \operatorname{tr}(M(x)D^2u)
  \quad\text{on }(0,T]\times B_1,
  \qquad
  u(0,\cdot)=\Phi,
  \qquad
  u=\Phi\quad\text{on }(0,T]\times\partial B_1,
\]
does not admit a solution in
$
  C([0,T]\times\overline B_1)\cap C^{1,2}((0,T]\times B_1).
$
Indeed, if such a solution existed, applying It\^o's formula with localization
inside $B_1$ would identify it with the Feynman--Kac value above and
hence with $\Phi$, contradicting
$
  \Phi\notin C^2(B_1).
$
Thus, already for a time-homogeneous two-dimensional equation, continuity and uniform ellipticity of
the coefficient matrix do not suffice for  classical
regularity.
\end{example}

\bibliographystyle{amsplain}
\bibliography{SDE_1D}

\end{document}